\documentclass[11pt]{amsart}

\usepackage{amsmath,amssymb}
\usepackage{url}
\usepackage{graphicx}
\usepackage{color}
\usepackage{tikz}
\usepackage{tikz-3dplot}
\usepackage[labelformat=empty]{subfig}
\usepackage[euler-digits]{eulervm}
\usepackage[all]{xy}
\usepackage{tikz}
\usetikzlibrary{matrix,arrows}
\usepackage{color}
\usepackage{soul}
\usepackage{enumerate}

\address{Department of Mathematics, Yale University, 10 Hillhouse Avenue, New Haven, CT 06511}
\author[T.Leake]{Timothy Leake}
\email{timothy.leake@yale.edu}
\author[D.Ranganathan]{Dhruv Ranganathan}
\email{dhruv.ranganathan@yale.edu}

\date{\today}

\newtheorem{theorem}{Theorem}

\newtheorem{lemma}[theorem]{Lemma}
\newtheorem{proposition}[theorem]{Proposition}
\newtheorem{definition}[theorem]{Definition}

\newtheorem{conjecture}[theorem]{Conjecture}

\newtheorem{example}[theorem]{Example}
\newtheorem{remark}[theorem]{Remark}

\newcommand{\ZZ} {{\mathbb Z}}

\newcommand{\cal}{\mathcal}

\def\cM{{\cal M}}

\newcommand{\Mbar}{\overline{\cM}}

\begin{document}

\title{Brill--Noether Theory of Maximally Symmetric Graphs}

\pagestyle{plain}

\begin{abstract}
We analyze the Brill--Noether theory of trivalent graphs and multigraphs having largest possible automorphism group in a fixed genus. For trivalent multigraphs with loops of genus at least $3$, we show that there exists a graph with maximal automorphism group which is Brill--Noether special. We prove similar results for multigraphs without loops of genus at least $6$, as well as simple graphs of genus at least $7$. This analysis yields counterexamples, in any sufficiently large genus, to a conjecture of Caporaso.
\end{abstract}

\maketitle

\section{Introduction}\label{sec: introduction}

A graph $\Gamma$ of genus $g$ is said to be \textbf{Brill--Noether general} if for all positive integers $r$ and $d$ such that the number
\[
\rho^r_d(g) = g-(r+1)(g-d+r)
\]
is negative, there exist no effective divisors on $\Gamma$ of degree $d$ and rank at least $r$. We begin by recalling the following conjecture of Caporaso~\cite[Conjecture 6.6(1)]{Cap11}.
\begin{conjecture}\label{conj: caporaso}
Assume $g\geq 2$. Let $\Gamma$ be a trivalent graph of genus $g$ with largest possible automorphism group (among trivalent graphs of genus $g$). Then $\Gamma$ is Brill--Noether general.
\end{conjecture}

\noindent The conjecture is motivated by an analogy between the moduli space $\Mbar_g$ of Deligne-Mumford stable curves and the moduli space $\cM_g^{trop}$ of tropical curves. The latter is the space parametrizing genus $g$ metric graphs with integer vertex-weights. The divisor theory of algebraic curves is related to that of graphs by means of Baker's Specialization Lemma~\cite{B08} and its generalizations~\cite{AC13}.

For fixed $g$, there is a natural order reversing bijection between the set of strata of these two moduli spaces. The classical moduli space $\Mbar_g$ is stratified by the combinatorial type of the dual graphs of stable curves, while $\cM_g^{trop}$ is stratified by the underlying combinatorial type of the metric graphs. That is, given a combinatorial graph $\Gamma$, there is a cone in the moduli space $\cM^{trop}$ parametrizing all metric graphs having $\Gamma$ as their underlying combinatorial graph. In particular, the top dimensional stratum on the classical side are the points of $\Mbar_g$ corresponding to smooth curves. The dual graph of a smooth genus $g$ curve is a single vertex with weight $g$. On the other hand, the top dimensional stratum on the tropical side corresponds to trivalent metric graphs with no nonzero vertex weights. A general smooth curve has no automorphisms, so Caporaso proposes that its analog should be a trivalent graph with the greatest possible number of automorphisms. Since a general curve is Brill--Noether general, Caporaso conjectures that the trivalent graphs with largest number of automorphisms are also Brill--Noether general. Here we prove that in any sufficiently large genus, there exists a trivalent graph (or multigraph) with largest possible automorphism group that is not Brill--Noether general.

In what follows, we will refer to trivalent graphs or multigraphs achieving largest possible automorphism group as \textbf{maximally symmetric}.

\begin{theorem}\label{thm: loops}
Let $g\geq 3$. There exists a maximally symmetric genus $g$ trivalent multigraph (possibly with loops) that is not Brill--Noether general.
\end{theorem}

Trivalent multigraphs with loops are the dual graphs of stable curves in the zero stratum of $\Mbar_g$. We may also analyze Caporaso's conjecture for dual graphs of curves which have only smooth components. These are multigraphs without loops.

\begin{theorem}\label{thm: multiple-edges}
Let $g\geq 6$. There exists a maximally symmetric genus $g$ trivalent multigraph with no loops that is not Brill--Noether general.
\end{theorem}

Finally, we analyze the Brill--Noether theory of maximally symmetric trivalent simple graphs. That is, graphs without loops or multiple edges.

\begin{theorem}\label{thm: simple-graphs}
Let $g\geq 7$. There exists a maximally symmetric genus $g$ trivalent simple graph that is not Brill--Noether general.
\end{theorem}

Our methods are explicit and combinatorial. We rely on the analysis of maximally symmetric trivalent graphs pursued by van Opstall and Veliche in~\cite{OV06,OV10}. In each sufficiently large genus, we exhibit a maximally symmetric graph together with a divisor of low degree and rank $1$. 

The preceding theorems are sharp. The maximally symmetric genus $2$ multigraph with loops is Brill--Noether general. When $g\leq 5$, all genus $g$ maximally symmetric multigraphs without loops are Brill--Noether general. The same is true for maximally symmetric simple graphs when $g\leq 6$.

\begin{remark}\label{rem: sharpness}\textnormal{
In~\cite[Proposition~6.5]{OV06} it is proved that when the genus $g$ is $3\cdot 2^m$, $3 (2^m+1)$, $3(2^m+2^p)$, or $3(2^m+2^p+1)$ (where $p>m\geq 0$), there is a unique maximally symmetric multigraph with no loops. A direct application of the proof of this result also shows that when $g$ is $3\cdot 2^m$, or $3 (2^m+2^p)$ there is a unique maximally symmetric multigraph with loops. In other words, in these genera, the theorems above preclude the existence of maximally symmetric Brill--Noether general graphs.}
\end{remark}

The heuristic used by van Opstall and Veliche is that maximally symmetric trivalent graphs should be ``as close to trees as possible''. As we will see in Section~\ref{sec: proof-of-theorems}, this is achieved by attaching trees to a small graph, such as a single vertex or a triangle, and placing appendages on the leaves of the trees to contribute genus. Our analysis suggests that graphs of this general structure are not Brill--Noether general for large genus. Together with Remark~\ref{rem: sharpness}, this leads us to pose the following conjecture.

\begin{conjecture}
Let $\Gamma$ be a genus $g$ trivalent maximally symmetric graph or multigraph (with or without loops). If $g$ is sufficiently large, then $\Gamma$ is not Brill--Noether general. 
\end{conjecture}

\subsection*{Acknowledgements}
It is a pleasure to thank Sam Payne for suggesting this project, as well as for his guidance and support. We are also grateful to Yoav Len and the anonymous referees for helpful comments and corrections on a previous version of this paper.

\section{Background}\label{sec: background}

\subsection{Divisors on Graphs}\label{subsec: divisors}
We now briefly recall the fundamentals of divisor theory on graphs and Brill--Noether theory. For further details, see~\cite{B08, BN09,Cap11}.

A graph $\Gamma$ will mean a finite connected graph possibly with loops and multiple edges. The vertex set of $\Gamma$ will be denoted by $V(\Gamma)$ and edge set $E(\Gamma)$. We will be explicit when restricting our analysis to the cases of graphs without multiple edges or loops. The genus of $\Gamma$, denoted $g(\Gamma)$, is the first Betti number of $\Gamma$. Since $\Gamma$ is a connected graph, we have
\[
g(\Gamma) = |E(\Gamma)|-|V(\Gamma)|+1.
\]

\noindent A \textbf{divisor} is an element of the free abelian group on the vertices of $\Gamma$:
\[
\textnormal{Div}(\Gamma) := \left\{\sum_{v\in V(\Gamma)} n_vv: n_v\in \ZZ\right\}.
\]

The \textbf{degree} of a divisor $D = \sum_v n_v v$ is $deg(D) = \sum_v n_v$. The divisor $D$ is said to be \textbf{effective} if $n_v\geq 0$ for all $v$. 

It is often helpful to visualize divisors on a graph $\Gamma$ as a configuration of ``chips'' and ``anti-chips'' placed on at vertices. Let $D$ be any divisor on $\Gamma$. A \textbf{chip firing move at $v$} produces a new divisor $D'$ as follows. Reduce the number of chips at $v$ by $deg(v)$, and add one chip to every vertex of $\Gamma$ adjacent to $v$. Note that the process of firing $v$ may be inverted by firing every vertex in the complement of $v$. Observe that a chip firing move does not change the degree of a divisor. 

\begin{definition}
Let $D$ and $D'$ be two divisors on a graph $\Gamma$. $D$ and $D'$ are said to be \textbf{linearly equivalent} if $D'$ can be obtained from $D$ by a finite sequence of chip firing moves at vertices of $\Gamma$. 
\end{definition}

For an alternate definition of this equivalence in terms of piecewise linear functions, see~\cite{Cap11}.

As with the divisor theory of algebraic curves, there is a natural notion of rank for divisors on graphs. For a divisor $D$, we denote by $|D|$ the set of \textbf{effective} divisors linearly equivalent to $D$. 

\begin{definition}
The \textbf{rank} of $D$, denoted $r(D)$, is defined as follows. If $|D| = \emptyset$, then $r(D) := -1$. Otherwise $r(D)$ is the largest integer $r$, such that $|D-E|$ is nonempty for all effective divisors $E$ of degree $r$. 
\end{definition}


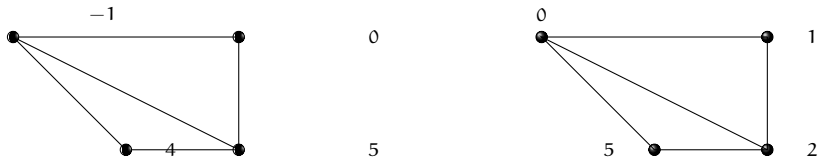
\begin{figure}
\begin{minipage}[b]{0.45\linewidth}
\begin{tikzpicture}[scale=1.5]
\hspace{12mm}
\begin{scope}[rotate=90]
\coordinate (0) at (0,0);
\coordinate (1) at (1,1);
\coordinate (2) at (1,-1);
\coordinate (3) at (0,-1);

\node (A) at (0,0.4) {\tiny$4$};
\node (B) at (1.2,1) { \tiny$ -1$};
\node (C) at (0,-1.4) { \tiny$ 5$};
\node (D) at (1,-1.4) {\tiny$0$};

\draw (0)--(1)--(3);\draw (0)--(3);\draw (3)--(2);\draw (2)--(1);

\draw[ball color=black] (0) circle (0.5mm);
\draw[ball color=black] (1) circle (0.5mm);
\draw[ball color=black] (2) circle (0.5mm);
\draw[ball color=black] (3) circle (0.5mm);
\end{scope}
\end{tikzpicture}

\end{minipage}
\quad
\begin{minipage}[b]{0.45\linewidth}
\hspace{8mm}
\begin{tikzpicture}[scale=1.5]
\begin{scope}[rotate=90]
\coordinate (0) at (0,0);
\coordinate (1) at (1,1);
\coordinate (2) at (1,-1);
\coordinate (3) at (0,-1);

\node (A) at (0,0.4) {\tiny$5$};
\node (B) at (1.2,1) { \tiny$ 0$};
\node (C) at (0,-1.4) { \tiny$ 2$};
\node (D) at (1,-1.4) {\tiny$1$};

\draw[ball color=black] (0) circle (0.5mm);
\draw[ball color=black] (1) circle (0.5mm);
\draw[ball color=black] (2) circle (0.5mm);
\draw[ball color=black] (3) circle (0.5mm);

\draw (0)--(1)--(3);\draw (0)--(3);\draw (3)--(2);\draw (2)--(1);
\end{scope}
\end{tikzpicture}
\end{minipage}
\caption{Chip firing. The divisor on the right is obtained by firing bottom right vertex on the left graph.}
\label{fig: chip-firing}
\end{figure}

\subsection{Brill--Noether Theory}\label{subset: BN-theory}
Let $C$ be a smooth projective algebraic curve. Brill--Noether theory of algebraic curves is concerned with the geometry of the scheme $W^r_d(C)$, parametrizing linear equivalence classes of divisors of degree $d$ moving in a linear system of rank at least $r$ on $C$. Brill--Noether theory of graphs seeks to develop an analogous combinatorial theory. One aspect of this is the existence of effective divisors of degree $d$ which have rank at least $r$ for positive integers $d$ and $r$.

Given a genus $g$ graph $\Gamma$ the \textbf{Brill--Noether number}, denoted $\rho^r_d(g)$, is defined to be
\[
\rho^r_d(g) := g-(r+1)(g-d+r).
\]
For an algebraic curve $C$, $\rho^r_d(g)$ is the naive expected dimension of $W^r_d(C)$. The classical \textbf{Brill--Noether theorem} states that for $\rho^r_d(g)<0$, a general curve of genus $g$ has no divisors of degree $d$ and rank at least $r$~\cite{GH80}. In~\cite{CDPR}, Cools, Draisma, Payne and Robeva prove a ``tropical Brill--Noether theorem'' which is then used to deduce the classical result. This tropical Brill--Noether theorem asserts the existence of a metric graph $\Gamma_g$ in each genus $g$ such that $\Gamma_g$ has no divisors of degree $d$ and rank at least $r$ when $\rho^r_d(g)$ is negative.
With this in mind, a genus $g$ graph $\Gamma$ is said to be \textbf{Brill--Noether general} if for all $r,d\geq0$ such that $\rho^r_d(g)$ is negative, $\Gamma$ has no divisors of degree $d$ and rank at least $r$. A graph is called \textbf{special} if it is not general.

\section{Proof of Theorems}\label{sec: proof-of-theorems}
\subsection{Maximally Symmetric Multigraphs with Loops}\label{sec: loops} All graphs in this subsection are trivalent multigraphs, possibly with loop edges. These are precisely the dual graphs of stable curves in the zero-stratum of $\Mbar_g$. In~\cite{OV06}, van Opstall and Veliche demonstrate explicit bounds for the automorphism group of a genus $g$ trivalent multigraph with loops together with a graph $C_g$ which achieves this bound. We now review the construction.\\

For $n\geq 3$ we construct trees $T_n$ and the genus $g$ graphs $C_g$ using the following algorithms given in \cite[Definition 3.1, 3.3]{OV06}.

\begin{definition}\label{def: Tn}
For each $n\geq 3$, $T_n$ is defined as follows:
\begin{enumerate}[(R1)]
\item Place $n$ vertices in a row. Call this level one.
\item Assume that level $k$ is formed, and contains vertices $v_1,\ldots, v_m$. If $m$ is even, form level $k+1$ by adding vertices labeled as follows: 
\[
v_{\{1,2\}},v_{\{3,4\}},\ldots,v_{\{m-1,m\}}.
\]
If $m$ is odd, form level $k+1$ by adding vertices labeled as follows:
\[
v_{\{1,2\}},v_{\{3,4\}},\ldots,v_{\{m-2,m-1\}}. 
\]
That is, level $k+1$ contains precisely $\lfloor \frac{m}{2}\rfloor$ vertices. Connect each vertex $v_{\{r,r+1\}}$ in level $k+1$ to the vertices $v_r$ and $v_{r+1}$ in level $k$.
\item Call a vertex \textbf{unpaired} if it is either $2$-valent or $0$-valent.
\item If at some stage there are exactly two unpaired vertices, connect them with an edge.
\item If at some stage there are exactly three unpaired vertices, connect all three by edges to a new vertex.
\item If there is an unpaired vertex in level $k$ and one in some level $\ell<k$, place a vertex in level $k+1$, and connect these two vertices to it. Note that this rule does not apply if there are \textbf{precisely} two unpaired vertices.
\end{enumerate}
Note that if both rules (R5) and (R6) apply, there is no conflict, as the resulting graphs are identical. The process terminates when the graph is connected with $n$ leaves, trivalent away from the leaves. 
\end{definition}

\begin{example}
\textnormal{
To illustrate this algorithm, we construct $T_5$. .
\begin{itemize}
\item Begin with vertices $v_1,\ldots, v_5$ at level one.
\item Form level two by $v_{\{1,2\}}$ and $v_{\{3,4\}}$. Connect $v_{\{i,j\}}$ to $v_i$ and $v_j$.
\item Level two is now formed. There are precisely three vertices remaining, namely $v_5$, $v_{\{1,2\}}$ and $v_{\{3,4\}}$. Following (R5) above, introduce a new vertex $v_0$ and connect it to each of these. The resulting graph is shown in Figure~\ref{fig: Tn}. Note that in place of (R5), one may also perform (R6) and then (R4), to obtain the same result.
\end{itemize}
}
\end{example}

\begin{figure}
\centering
\begin{minipage}[b]{0.45\linewidth}
\hspace{13mm}
\begin{tikzpicture}
[scale=1.5]
\begin{scope}[rotate=90]
\coordinate (0) at (0,0);
\coordinate (1) at (-0.5,-0.5);
\coordinate (2) at (0.5,-0.5);
\coordinate (3) at (-0.75,-1);
\coordinate (4) at (-0.25,-1);
\coordinate (5) at (0.25,-1);
\coordinate (6) at (0.75,-1);
\coordinate (7) at (0,0.7);

\draw[ball color=black] (0) circle (0.3mm);
\draw[ball color=black] (1) circle (0.3mm);
\draw[ball color=black] (2) circle (0.3mm);
\draw[ball color=black] (3) circle (0.55mm);
\draw[ball color=black] (4) circle (0.55mm);
\draw[ball color=black] (5) circle (0.55mm);
\draw[ball color=black] (6) circle (0.55mm);
\draw[ball color=black] (7) circle (0.55mm);
\draw (0)--(7);
\draw (0)--(1); \draw (0)--(2); \draw (1)--(3);\draw (1)--(4);\draw (2)--(6);\draw (2)--(5);

\end{scope}
\end{tikzpicture}

\caption*{$T_5$.}

\end{minipage}
\quad
\begin{minipage}[b]{0.45\linewidth}
\centering
\begin{tikzpicture}
[scale=1.5]
\begin{scope}[rotate=90]
\coordinate (0) at (0,0);
\coordinate (1) at (-0.5,-0.5);
\coordinate (2) at (0.5,-0.5);
\coordinate (3) at (-0.75,-1);
\coordinate (4) at (-0.25,-1);
\coordinate (5) at (0.25,-1);
\coordinate (6) at (0.75,-1);
\coordinate (7) at (0,0.7);
\coordinate (8) at (0.25,1.2);
\coordinate (9) at (-0.25,1.2);

\draw[ball color=black] (0) circle (0.3mm);
\draw[ball color=black] (1) circle (0.3mm);
\draw[ball color=black] (2) circle (0.3mm);
\draw[ball color=black] (3) circle (0.55mm);
\draw[ball color=black] (4) circle (0.55mm);
\draw[ball color=black] (5) circle (0.55mm);
\draw[ball color=black] (6) circle (0.55mm);
\draw[ball color=black] (7) circle (0.3mm);
\draw[ball color=black] (8) circle (0.55mm);
\draw[ball color=black] (9) circle (0.55mm);
\draw (0)--(7);
\draw (0)--(1); \draw (0)--(2); \draw (1)--(3);\draw (1)--(4);\draw (2)--(6);\draw (2)--(5);
\draw (9)--(7)--(8);
\end{scope}
\end{tikzpicture}

\caption*{$T_6$.}
\end{minipage}
\caption{The trees $T_5$ and $T_6$. In each case, the larger vertices correspond to the level one vertices.}
\label{fig: Tn}

\end{figure}

\begin{definition}\label{def: Cn}
For each $g\geq 3$, the graph $C_g$ is defined as follows:
\begin{itemize}
\item If $g\neq 3\cdot 2^m+1$, for $m>0$ and $g \neq 3(2^m+2^p)$ for $m>p\geq 0$, form $C_g$ from $T_g$ by placing a loop on each leaf. 
\item If $g = 3\cdot 2^m+1$, then form $C_{g}$ by following the construction for $T_{g-1}$, but at the last step join the last three vertices in a triangle, then placing a loop on each leaf.
\item If $g = 3(2^m+2^p)$ for $p\geq 0$ and $m>p+1$, the graph $C_g$ is constructed as follows. Let $B$ be the line graph on $3$ vertices. To one leaf of $B$, attach a copy of $T_{2^m}$, and to the other leaf, a copy of $T_{2^p}$. The resulting graph has a unique $2$-valent vertex, namely that of $B$. Take three copies of this graph and attach them to the edges of a $3$-star. The graphs are attached to the $3$-star at their unique $2$-valent vertices. Around each of the $3(2^m+2^p)$ leaves of this graph, place a loop. This is $C_g$.
\end{itemize}
\end{definition}

\noindent The following proposition is a direct consequence of Proposition~3.4 and of the Main Theorem of~\cite{OV06}.

\begin{proposition}
Let $g\geq 3$. The graph $C_g$ is a maximally symmetric genus $g$ trivalent multigraph with loops.
\end{proposition}

\begin{figure}[h!]
\small
\newsavebox{\tempbox}
\sbox{\tempbox}{
\begin{tikzpicture}
[scale=1.2]
\coordinate (0) at (0,0);
\coordinate (1) at (-0.5,-0.5);
\coordinate (2) at (0.5,-0.5);
\coordinate (3) at (-0.75,-1);
\coordinate (4) at (-0.25,-1);
\coordinate (5) at (0.25,-1);
\coordinate (6) at (0.75,-1);

\draw[ball color=black] (1) circle (0.3mm);
\draw[ball color=black] (2) circle (0.3mm);
\draw[ball color=black] (3) circle (0.3mm);
\draw[ball color=black] (4) circle (0.3mm);
\draw[ball color=black] (5) circle (0.3mm);
\draw[ball color=black] (6) circle (0.3mm);
\draw (1)--(3);\draw (1)--(4);\draw (2)--(6);\draw (2)--(5);
\draw (-0.75, -1.15) circle (0.15); \draw (-0.25, -1.15) circle (0.15);
\draw (0.25,-1.15) circle (0.15); \draw (0.75,-1.15) circle (0.15);
\draw (1)--(2);
\end{tikzpicture}
}
\newlength{\tempboxlength}
\settowidth{\tempboxlength}{\usebox{\tempbox}}
\subfloat[ $C_4$]{
	\usebox{\tempbox}
}
\makebox[\tempboxlength]{
\hspace{1mm}
\subfloat[$C_5$]{
\begin{tikzpicture}
[scale=1.2]
\begin{scope}[rotate=90]
\coordinate (0) at (0,0);
\coordinate (1) at (-0.5,-0.5);
\coordinate (2) at (0.5,-0.5);
\coordinate (3) at (-0.75,-1);
\coordinate (4) at (-0.25,-1);
\coordinate (5) at (0.25,-1);
\coordinate (6) at (0.75,-1);
\coordinate (7) at (0,0.7);

\draw[ball color=black] (0) circle (0.3mm);
\draw[ball color=black] (1) circle (0.3mm);
\draw[ball color=black] (2) circle (0.3mm);
\draw[ball color=black] (3) circle (0.3mm);
\draw[ball color=black] (4) circle (0.3mm);
\draw[ball color=black] (5) circle (0.3mm);
\draw[ball color=black] (6) circle (0.3mm);
\draw[ball color=black] (7) circle (0.3mm);
\draw (0)--(7); \draw (0,0.85) circle (0.15);
\draw (0)--(1); \draw (0)--(2); \draw (1)--(3);\draw (1)--(4);\draw (2)--(6);\draw (2)--(5);
\draw (-0.75, -1.15) circle (0.15); \draw (-0.25, -1.15) circle (0.15);
\draw (0.25,-1.15) circle (0.15); \draw (0.75,-1.15) circle (0.15);
\end{scope}
\end{tikzpicture}
}
\hfill
}
\hspace{5mm}
\subfloat[$C_7$]{
\begin{tikzpicture}
[scale = 2]
\begin{scope}[rotate=90]
\coordinate (0) at (0,0.25);
\coordinate (1) at (-0.15,0);
\coordinate (2) at (0.15,0);
\coordinate (3) at (0,0.45);
\coordinate (4) at (-0.3,-0.15);
\coordinate (5) at (0.3,-0.15);
\coordinate (6) at (-0.125,0.6); \coordinate (7) at (0.125,0.6);
\coordinate (8) at (-0.3,-0.35);
\coordinate (9) at (-0.49,-0.15);
\coordinate (10) at (0.49,-0.15);
\coordinate (11) at (0.3,-0.35);

\draw[ball color=black] (0) circle (0.2mm);
\draw[ball color=black] (1) circle (0.2mm);
\draw[ball color=black] (2) circle (0.2mm);
\draw[ball color=black] (3) circle (0.2mm);
\draw[ball color=black] (4) circle (0.2mm);
\draw[ball color=black] (5) circle (0.2mm);
\draw[ball color=black] (6) circle (0.2mm);
\draw[ball color=black] (7) circle (0.2mm);
\draw[ball color=black] (8) circle (0.2mm);
\draw[ball color=black] (9) circle (0.2mm);
\draw[ball color=black] (10) circle (0.2mm);
\draw[ball color=black] (11) circle (0.2mm);
\draw (0)--(1)--(2)--(0);\draw (0)--(3);\draw (1)--(4);\draw (2)--(5);\draw (6)--(3)--(7);
\draw (8)--(4)--(9); \draw (10)--(5)--(11);
\draw(-0.125,0.69) circle (0.09); \draw(0.125,0.69) circle (0.09);
\draw (-0.3,-0.44) circle (0.09); \draw (-0.58, -0.15) circle (0.09);
\draw (0.3,-0.44) circle (0.09); \draw (0.58, -0.15) circle (0.09);
\end{scope}
\end{tikzpicture}

\hfill
}

\caption{\small{The hyperelliptic graphs $C_4$,$C_5$ and $C_7$. The hyperelliptic involution on $C_4$ and $C_5$ flips each loop. $C_7$ is not hyperelliptic}}
\label{fig: Cn}
\end{figure}
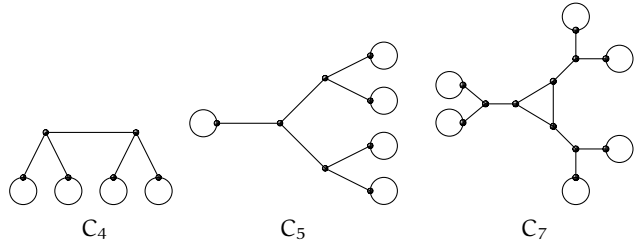

\begin{definition}
A graph $\Gamma$ is called hyperelliptic it has a divisor of rank $1$ and degree $2$.
\end{definition}

Similarly, an algebraic curve is \textbf{hyperelliptic} if it has a divisor of rank $1$ and degree $2$. An algebraic curve is hyperelliptic if and only if it admits an involution $\iota$ whose quotient is a genus $0$ curve. In similar spirit, Baker and Norine show in~\cite[Theorem 5.2]{BN09} that a loop-free graph $\Gamma$ is hyperelliptic if and only if there exists an involution $\iota: \Gamma\to \Gamma$, such that $\Gamma/\iota$ is a tree. If $\iota$ exists, it is necessarily unique, and we refer to it as the \textbf{hyperelliptic involution.} The divisor of rank $1$ and degree $2$ is unique up to linear equivalence, and is referred to as the \textbf{hyperelliptic divisor}. It was shown by Caporaso~\cite[Lemma~4.1]{Cap12} that any graph $\Gamma$ is hyperelliptic if and only if that graph $\Gamma^0$, obtained by subdividing every loop edge, is hyperelliptic.

For $g$ at least $3$, hyperelliptic graphs are not Brill--Noether general, since $\rho^1_2(g)$ is negative. We will use this fact repeatedly in the forthcoming proofs.\\

\noindent \textit{Proof of Theorem~\ref{thm: loops}.}
We see immediately from the construction above that if $g \neq 3\cdot 2^m+1$, the graph $C_g$ is formed by attaching loops to the leaves of a tree. Form $C_g^0$ by inserting a $2$-valent vertex in the center of every loop. The automorphism that interchanges the edges of this subdivided loop is a hyperelliptic involution, so $C^0_g$ is hyperelliptic. By the result of Caporaso discussed above, $C_g$ is also hyperelliptic and consequently not Brill--Noether general for $g\geq 3$.

Now consider the case where $g =3\cdot 2^m+1$. The graph $C_g$ is formed from a triangle by attaching to each vertex (via an edge) hyperelliptic subgraphs of genus $2^m$, as in Definition~\ref{def: Cn}. Consider the divisor $D=3v$ where $v$ is a vertex of the central triangle. We want to show that $D$ is a special divisor degree $3$ and rank $1$.

To see that $D$ has rank $1$, we may assume that $D$ is zero away from the vertices of the central triangle by using a sequence of chip firing moves to move any nonzero weight onto the triangle. This is a consequence of the following standard argument: suppose $e$ is a bridge edge\footnote{A bridge edge of a connected graph $\Gamma$ is an edge $e$ such that if $e$ is deleted, $\Gamma$ becomes disconnected.} with vertices $v$ and $v'$. Then firing every vertex in the connected component of $\Gamma\setminus e$ that contains $v$ moves $v$ to $v'$, so $v\sim v'$. It follows that any two vertices connected by a bridge edge are equivalent. The problem is now reduced the case of a divisor $3w-w'$ where $w,w'$ are vertices of the central triangle. Since $\Gamma$ is trivalent, firing $w$ yields an effective divisor. Since $\rho^1_3(g)<0$, $D$ is a special divisor and $C_g$ is not Brill--Noether general.
\qed

\begin{remark}\textnormal{
Let $D$ be a special divisor on a graph $\Gamma$. We may ask if $D$ is the specialization of a divisor on an algebraic curve in the sense of Baker~\cite{B08}. Loosely speaking, we ask whether $\Gamma$ is Brill--Noether special for algebraic reasons rather than combinatorial ones.
When $g\neq 3\cdot 2^m+1$, the hyperelliptic divisor is not the specialization of such a divisor on an algebraic curve. This follows from the fact that the hyperelliptic involution on this graph cannot be lifted to a hyperelliptic involution on an algebraic curve. See~\cite[Corollary 4.15]{ABBR}.}
\end{remark}

\subsection{Maximally Symmetric Mulitgraphs without Loops}\label{multi}
All graphs in this subsection are trivalent with no loop edges, but may have multiple edges. We first describe graphs $C'_g$ for all genera, following~\cite[Definition~3.5]{OV06}. By a \textbf{cone} we mean a triangle with one edge doubled. Note that a cone has genus 2.

The construction again uses as an ingredient the graphs $T_n$ from Definition~\ref{def: Tn}.

\begin{definition}For each $g\geq 3$, the graph $C_g'$ is defined as follows:
\begin{itemize}
\item If $g = 3\cdot 2^m$, $m>0$, $C'_g$ is formed by three copies of the graph$T_{2^{m-1}}$, connected in a 3-star. A cone is placed on each leaf. 
\item If $g = 3\cdot 2^m+1$, $m>0$, $C'_g$ is formed from $C'_{g-1}$ by expanding the central vertex to a triangle.
\item If $g = 3\cdot 2^m+2$,  $C'_g$ is formed from $C'_{g-2}$ by expanding the central vertex to a $K_{2,3}$\footnote{Recall that $K_{2,3}$ is the complete bipartite graph. It has 2 trivalent vertices and 3 bivalent vertices. The trees are attached to these bivalent vertices.}.
\item If $g = 3\cdot(2^m+1)$, $m>1$, $C'_g$ is formed from $C'_{g-3}$ by adding double edges in the middle of the edges of the star.

\item If $g = 3(2^m+2^p)$ for $p> 0$ and $m>p+1$, the graph $C'_g$ is constructed as follows. Let $B$ be the line graph on $3$ vertices. To one leaf of $B$, attach a copy of $T_{2^{m-1}}$, and to the other leaf, a copy of $T_{2^{p-1}}$. The resulting graph has a unique $2$-valent vertex, namely that of $B$. Take three copies of this graph and attach them to the edges of a $3$-star. The graphs are attached to the $3$-star at their unique $2$-valent vertices. At each of the $3(2^m+2^p)$ leaves of this graph, place a cone. This is $C'_g$.

\item If $g = 3(2^m+2^p)$, $m>p+1$ and $p>0$, insert double edges to the previous case in the edges emanating from the star.
\item Otherwise if $g$ is even, form $C_{g/2}$ as in Section~\ref{sec: loops}, replacing loops by cones.
\item If $g$ is odd, then the last step of the construction of $T_{\lfloor g/2 \rfloor}$ (the tree in Section~\ref{sec: loops}) ends by connecting $2$ vertices by an edge or $3$ by a star. In the former case, add a double edge in the middle of this edge. In the latter case, the subgraphs on one of the 3 branches of the star is not isomorphic to the other two. Add a double edge on this branch. This is $C'_g$.
\end{itemize}
\end{definition}

\begin{figure}[h!]
\centering
\begin{minipage}[b]{0.45\linewidth}
\hspace{13mm}
\begin{tikzpicture}[scale = 1]
\begin{scope}
\coordinate (0) at (0,0);
\coordinate (1) at (1,0);
\coordinate (2) at (-1,0);
\coordinate (3) at (0,1);
\coordinate (4) at (-1.2,-0.6);
\coordinate (5) at (-0.8,-0.6);
\coordinate (6) at (1.2,-0.6);
\coordinate (7) at (0.8,-0.6);
\coordinate (8) at (0.2, 1.6);
\coordinate (9) at (-0.2, 1.6);

\draw[ball color=black] (0) circle (0.5mm);
\draw[ball color=black] (1) circle (0.5mm);
\draw[ball color=black] (2) circle (0.5mm);
\draw[ball color=black] (3) circle (0.5mm);
\draw[ball color=black] (4) circle (0.5mm);
\draw[ball color=black] (5) circle (0.5mm);
\draw[ball color=black] (6) circle (0.5mm);
\draw[ball color=black] (7) circle (0.5mm);
\draw[ball color=black] (8) circle (0.5mm);
\draw[ball color=black] (9) circle (0.5mm);

\draw (1)--(0)--(2); \draw (0)--(3);
\path (4) edge [bend left] (5);
\path (4) edge [bend right] (5);
\draw (4)--(2)--(5);
\path (6) edge [bend left] (7);
\path (6) edge [bend right] (7);
\draw (6)--(1)--(7);
\path (8) edge [bend left] (9);
\path (8) edge [bend right] (9);
\draw (9)--(3)--(8);
\end{scope}
\end{tikzpicture}
\hspace{5mm}
\caption*{The multigraph $C'_6$}

\end{minipage}
\quad
\begin{minipage}[b]{0.45\linewidth}
\hspace{14mm}
\begin{tikzpicture}
\coordinate (0) at (-1,0);
\coordinate (1) at (-0.5, 0.5);
\coordinate (2) at (0.5, 0.5);
\coordinate (3) at (1,0);
\coordinate (4) at (0.5, -0.5);
\coordinate (5) at (-0.5, -0.5);

\draw[ball color=black] (0) circle (0.5mm);
\draw[ball color=black] (1) circle (0.5mm);
\draw[ball color=black] (2) circle (0.5mm);
\draw[ball color=black] (3) circle (0.5mm);
\draw[ball color=black] (4) circle (0.5mm);
\draw[ball color=black] (5) circle (0.5mm);

\draw (0)--(1);\path (1) edge [bend left] (2); \path (1) edge [bend right] (2);
\draw (2)--(3);\path (3) edge [bend left] (4); \path (3) edge [bend right] (4);
\draw (4)--(5);\path (5) edge [bend left] (0); \path (0) edge [bend left] (5);
\end{tikzpicture}

\caption*{The graph $\tilde C_4$}
\end{minipage}
\caption{While the graph on the left is hyperelliptic, the graph on the right is not.}
\label{fig: C'6-multigraph}

\end{figure}

\noindent The following proposition is a direct consequence of Proposition~3.6 and Main Theorem of~\cite{OV06}.

\begin{proposition}
Let $g\geq 10$. The graph $C_g'$ is a maximally symmetric genus $g$ trivalent multigraph without loops.
\end{proposition}

\subsubsection{Low Genus Cases} We first analyze Caporaso's conjecture for multigraphs without loops in genus smaller than $10$. The following results are obtained by explicitly determining the divisor classes of prescribed rank and degree on a given graph using E. Robeva's program\footnote{This program efficiently computes the number of degree $d$ rank $r$ divisors on a given simple graph. The pseudocode can be found at \url{http://math.berkeley.edu/~erobeva/about.html}}.

We denote by $\tilde C_g$, the genus $g$ ``loop of loops'' graph. Namely, the graph obtained from the $2(g-1)$-gon by doubling every other edge.
\begin{itemize}
\item In genus $3$ the unique maximally symmetric multigraph is in fact simple. It is the tetrahedron, and is Brill--Noether general.
\item In genus $4$ the unique maximally symmetric multigraph is also simple. It is the complete bipartite graph$K_{3,3}$, and is Brill--Noether general.
\item In genus $5$, the unique maximally symmetric multigraph is $\tilde C_5$. This graph is Brill--Noether general. 
\item In genus $6$ and $8$, the unique maximally symmetric multigraphs are $C'_6$ and $C'_8$ respectively. As we will see, these graphs are hyperelliptic and hence special.
\item In genus $7$, the unique maximally symmetric multigraph is $\tilde C_7$. This graph has an effective divisor of degree $4$ and rank $1$, and hence is special. The special divisor is $D = 2v+2w$, where $v$ and $w$ are any two vertices connected by a doubled edge.
\item In genus $9$, the maximally symmetric multigraphs are $C'_9$ and $\tilde C_9$. These graphs are both Brill--Noether special. The former is hyperelliptic, and the latter has a divisor of degree $4$ and rank $1$. Again, the special divisor is $D = 2v+2w$ where $v$ and $w$ are connected by a doubled edge.
\end{itemize}

\subsubsection{Higher Genus Cases: Proof of Theorem~\ref{thm: multiple-edges}}
If $g \neq 3\cdot 2^m+1$, $C'_g$ is hyperelliptic. The hyperelliptic involution reflects all cones, and interchanges the edges on all double edges. If the graph contains a $K_{2,3}$ subgraph, the hyperelliptic involution is as depicted in Figure~\ref{fig: K23}.

If $g = 3\cdot 2^m+1$, $C'_g$ has a divisor of degree $3$ and rank $1$. The argument is nearly identical to that used for multigraphs with loops in the previous section. The relevant divisor in this case is $3v,$ where $v$ is the top vertex of any cone.  \qed

\begin{figure}
\centering
\begin{minipage}[b]{0.45\linewidth}
\hspace{13mm}
\begin{tikzpicture}
\coordinate (0) at (-1,0);
\coordinate (1) at (1,0);
\coordinate (2) at (0,0.33);
\coordinate (3) at (0,0.67);
\coordinate (4) at (0,1);

\node (N0) at (-1,-0.25) {$v$};\node (N1) at (1,-0.25) {$w$};

\draw[ball color=black] (0) circle (0.5mm);
\draw[ball color=black] (1) circle (0.5mm);
\draw[ball color=black] (2) circle (0.5mm);
\draw[ball color=black] (3) circle (0.5mm);
\draw[ball color=black] (4) circle (0.5mm);

\draw (0)--(2)--(1); \draw (0)--(3)--(1);\draw (0)--(4)--(1);

\end{tikzpicture}

\caption*{$K_{2,3}$.}

\end{minipage}
\quad
\begin{minipage}[b]{0.45\linewidth}
\centering

\begin{tikzpicture}
\coordinate (0) at (-1,0);
\coordinate (1) at (1,0);
\coordinate (2) at (0,1);
\node (N0) at (-1,-0.25) {$v$};\node (N1) at (1,-0.25) {$w$};

\draw[ball color=black] (0) circle (0.5mm);
\draw[ball color=black] (1) circle (0.5mm);
\draw[ball color=black] (2) circle (0.5mm);

\draw (0)--(2)--(1);
\path (0) edge [bend left] (1);
\path (0) edge [bend right] (1);

\end{tikzpicture}
\caption*{Cone.}
\end{minipage}
\caption{In each case the hyperelliptic involution interchanges the vertices $v$ and $w$, while fixing all other vertices.}
\label{fig: K23}

\end{figure}

\subsection{Maximally Symmetric Simple Graphs}\label{simple}
All graphs in this subsection are simple and trivalent. We begin with an analysis of the conjecture for low genus. 

\subsubsection{Low Genus Cases}
The following are the maximally symmetric trivalent graphs in each genus from~\cite{OV10} Remark 4.3.\footnote{In the $g=7$ case, the optimal graph was produced via an exhaustive search of all trivalent graphs in this genus using data from Gordon Royle available at \url{http://school.maths.uwa.edu.au/~gordon/data.html}.}
\begin{itemize}
\item In genus $3$, the unique maximally symmetric graph is the tetrahedron. It is Brill--Noether general.
\item In genus $4$, the unique maximally symmetric graph is $K_{3,3}$, the complete bipartite graph. It is Brill--Noether general.
\item In genus $5$, the unique maximally symmetric graph is the cube. It is Brill--Noether general.
\item In genus $6$, the unique maximally symmetric graph is the Petersen graph. It is Brill--Noether general.
\item In genus $7$, the unique maximally symmetric graph is depicted in Figure~\ref{fig: genus-7-simple}. The divisor depicted in this figure has degree $4$ and rank $1$, and the graph is hence special. 
\item In genus $8$, the unique maximally symmetric graph is the Heawood graph. This graph has an effective divisor of degree $7$ and rank $2$ and is hence special. 
\end{itemize}

\begin{figure}[h!]
\begin{tikzpicture}
[scale = 2]
\coordinate (0) at (-.5,0);
\coordinate (1) at (.5,0);
\coordinate (2) at (.75,.33);
\coordinate (3) at (.75,.66);
\coordinate (4) at (.5,1);
\coordinate (5) at (-.5,1);
\coordinate (6) at (-.75,.66);
\coordinate (7) at (-.75,.33);
\coordinate (8) at (1,.33);
\coordinate (9) at (1,.66);
\coordinate (10) at (-1,.66);
\coordinate (11) at (-1,.33);

\node (N0) at (-1.2,.25) {\small{$3$}};\node (N1) at (.6,-.14) {\small{$1$}};

\draw[ball color=black] (0) circle (0.3mm);
\draw[ball color=black] (1) circle (0.3mm);
\draw[ball color=black] (2) circle (0.3mm);
\draw[ball color=black] (3) circle (0.3mm);
\draw[ball color=black] (4) circle (0.3mm);
\draw[ball color=black] (5) circle (0.3mm);
\draw[ball color=black] (6) circle (0.3mm);
\draw[ball color=black] (7) circle (0.3mm);
\draw[ball color=black] (8) circle (0.3mm);
\draw[ball color=black] (9) circle (0.3mm);
\draw[ball color=black] (10) circle (0.3mm);
\draw[ball color=black] (11) circle (0.3mm);
\draw (0)--(1); \draw (1)--(2); \draw (2)--(3);\draw (3)--(4);\draw (4)--(5);\draw (5)--(6);\draw (6)--(7);\draw (7)--(0);\draw (1)--(8);\draw (8)--(9);\draw (9)--(4);\draw (5)--(10);\draw (10)--(11);\draw (11)--(0);\draw (2)--(9);\draw (3)--(8);\draw (7)--(10);\draw (6)--(11);
\end{tikzpicture}
\caption{The maximally symmetric simple graph of genus 7. The chip configuration depicted corresponds to a special divisor of degree 4 and rank 1.}
\label{fig: genus-7-simple}
\end{figure}
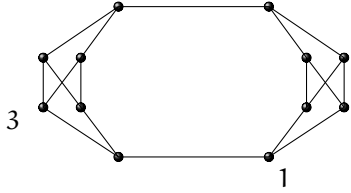

We observe that genus $8$ is the only case we find where a maximally symmetric graph has Brill--Noether general gonality, but has special divisors of higher rank. The divisor $7v$ for any vertex $v$ is a rank $2$ special divisor. 

\begin{remark}\textnormal{
In recent work, Jensen~\cite{J14} showed that the Haewood graph is a Brill--Noether special \textbf{metric} graph for arbitrary edge lengths, thus showing that the locus of Brill--Noether general metric graphs is not dense in the tropical moduli space.}
\end{remark}

\subsubsection{Higher Genus Cases}

In this section, we prove that for genus at least $9$, the maximally symmetric trivalent graph is not Brill--Noether general. We extensively use the analysis of maximally symmetric trivalent graphs in~\cite{OV10}. The proof demands some technical lemmas regarding the equivalence of divisors on certain types of graphs.

By a \textbf{pinched tetrahedron} or a \textbf{pinched $K_{3,3}$} we mean the graphs depicted in Figure~\ref{fig: pinched}. These are obtained from the standard tetrahedron or $K_{3,3}$ by adding a new vertex to subdivide an edge. 

\begin{figure}[h!]
\centering
\begin{minipage}[b]{0.45\linewidth}
\hspace{13mm}
\begin{tikzpicture}
\coordinate (0) at (0,1);
\coordinate (1) at (-1,0);
\coordinate (2) at (0,-1);
\coordinate (3) at (1,0);
\coordinate (4) at (0,0);

\draw[ball color=black] (0) circle (0.5mm);
\draw[ball color=black] (1) circle (0.5mm);
\draw[ball color=black] (2) circle (0.5mm);
\draw[ball color=black] (3) circle (0.5mm);
\draw[ball color=black] (4) circle (0.5mm);
\draw (0)--(1)--(2)--(3)--(0); \draw (4)--(0); \draw (4)--(1); \draw (4)--(3);
\end{tikzpicture}
\caption*{Pinched Tetrahedron}

\end{minipage}
\quad
\begin{minipage}[b]{0.45\linewidth}
\hspace{10mm}
\begin{tikzpicture}
\coordinate (0) at (-1.25,0);
\coordinate (1) at (-0.75,0.5);
\coordinate (2) at (0.75,0.5);
\coordinate (3) at (1.25,0);
\coordinate (4) at (-0.75,-0.5);
\coordinate (5) at (0.75,-0.5);
\coordinate (6) at (0,-2);

\draw[ball color=black] (0) circle (0.5mm);
\draw[ball color=black] (1) circle (0.5mm);
\draw[ball color=black] (2) circle (0.5mm);
\draw[ball color=black] (3) circle (0.5mm);
\draw[ball color=black] (4) circle (0.5mm);
\draw[ball color=black] (5) circle (0.5mm);
\draw[ball color=black] (6) circle (0.5mm);

\draw (0)--(1)--(2)--(3)--(5)--(4)--(0);\draw (1)--(5);\draw (2)--(4);\draw (0)--(6)--(3);
\end{tikzpicture}
\caption*{Pinched $K_{3,3}$}
\end{minipage}
\caption{}
\label{fig: pinched}

\end{figure}

\begin{definition}[{\cite[Definition~4.1]{OV10}}]
Let $A_m$ denote the graph formed by attaching a pinched tetrahedron (at its unique bivalent vertex) to each leaf of the tree $T_{2^m}$. Similarly, let $B_m$ ($m\geq 2$) denote the graph formed by attaching a pinched $K_{3,3}$ to each leaf of the tree $T_{2^{m-2}}$. Observe that $A_m$ has genus $3\cdot 2^m$ and $B_m$ has genus $2^m$.
\end{definition}

For $\Gamma = A_m$ or $B_m$, we let $T$ denote the subtree of $\Gamma$ obtained by discarding the attached pinched $K_{3,3}$'s or tetrahedra. Let $K$ be a pinched $K_{3,3}$ or tetrahedron. We abuse notation slightly, allowing $K$ to be ambiguous in denoting these subgraphs. $T$ is attached to each $K$ precisely one leaf vertex. The following lemma asserts that any two divisors of the same degree that are supported on $T$ are equivalent.

\begin{lemma}\label{lem1}
Let $\Gamma = A_m$ or $B_m$ be as above. Let $D$ and $D'$ be divisors on $\Gamma$ of the form $D=nv$ and $D'=nw$, $n\in \mathbb{Z}$ and $v,w\in V(T)$. Then $D\sim D'$.
\end{lemma}

\textit{Proof.} The proof of this result is similar to that of Theorem~\ref{thm: loops}. As there is a unique path of bridge edges between any two vertices of $T$, it suffices to show that $v\sim v'$, where $v$ and $v'$ are vertices connected by a single bridge edge. We may assume that $v'$ is closer to the root vertex of $T$ than $v$ is. There is a unique closest pinched $K_{3,3}$ or tetrahedron of $\Gamma$ to $v$. Denote this subgraph $K$. Firing each vertex in $K$ as well as those vertices in $T$ between $K$ and $v$ (including $v$) results in the cancelation of each chip except the one that arrives at $v'$, proving the claim. \qed

The following lemma will provide a source of special divisors in the proof of the main theorem.

\begin{lemma}\label{lem2}
Let $v$ be the unique vertex shared by $K$ and $T$. Suppose $D=4v$ and $E=-w$ for some vertex $w\in K$. Then $D+E$ is equivalent to an effective divisor.
\end{lemma}

\begin{proof}
Let $v$ be the unique vertex shared by $T$ and $K$, and let $v_1$ and $v_2$ be the vertices of $K$ that are adjacent to $v$. Let $u$ be the unique vertex of $T$ adjacent to $v$. By Lemma~\ref{lem1}, the chips on the vertex $u$ can always be moved to $v$, and thus we may reduce to working solely with the graph $K$, having $v$ as its unique $2$-valent vertex. It suffices to show that $4v-w$ is equivalent to effective. If $w$ is either $v_1$ or $v_2$, we may fire $v$ to obtain an effective divisor. Otherwise, fire $v$ three times, and then fire $v_1$ and $v_2$ to obtain an effective divisor.
\end{proof}

\begin{figure}
\begin{tikzpicture}
\coordinate (0) at (-1.25,0);
\coordinate (1) at (-0.75,0.5);
\coordinate (2) at (0.75,0.5);
\coordinate (3) at (1.25,0);
\coordinate (4) at (-0.75,-0.5);
\coordinate (5) at (0.75,-0.5);
\coordinate (6) at (0,-2);
\coordinate (7) at (0,-3);

\draw (0)--(1)--(2)--(3)--(5)--(4)--(0);\draw (1)--(5);\draw (2)--(4);\draw (0)--(6)--(3);

\draw[ball color=black] (0) circle (0.5mm);
\draw[ball color=black] (1) circle (0.5mm);
\draw[ball color=black] (2) circle (0.5mm);
\draw[ball color=black] (3) circle (0.5mm);
\draw[ball color=black] (4) circle (0.5mm);
\draw[ball color=black] (5) circle (0.5mm);
\draw[ball color=black] (6) circle (0.5mm);

\draw[densely dotted] (7)--(0,-2.5);
\draw (0,-2.5)--(6);

\node at (0.5,-2) {\tiny $4(v)$};
\node at (0.75,0.8) {\tiny $-1(w)$};
\node at (-1.55,0) {\tiny $v_2$};
\node at (1.55,0) {\tiny $v_1$};

\end{tikzpicture}
\caption{The divisor $D+E$ in Lemma~\ref{lem2}.}
\end{figure}
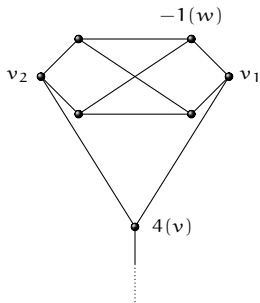

The graphs $A_m$ and $B_m$ have unique bivalent vertices, namely the roots of their subtrees. For certain $g$, maximally symmetric graphs can be constructed by attaching copies of the $A_m$ or $B_m$ to cycle graphs $G_n$ $(n\geq 3)$. This is done by introducing an edge between the unique bivalent vertices of the $A_m$ or $B_m$ graphs, and the vertices of the cycle. Consequently, the resulting graph is trivalent.

\begin{lemma}\label{lem3}
Let $\Gamma$ be defined by attaching a copy of $A_m$ or $B_m$ to each vertex  $v_1,...,v_n$ in a cycle graph $G_n$ of length $n$. Let $D=k\cdot nv_i$ and $D' = k\cdot nv_j$, $k\in \mathbb{Z}$. Then $D\sim D'$.
\end{lemma}

\begin{proof} 
Let $u_i$ be the unique vertex adjacent to $v_i$ that does not lie on the cycle $G_n$. As in the proof of Lemma~\ref{lem1}, since $u_i$ and $v_i$ are connected by a bridge edge, we know that $u_i\sim v_i$. As such, any chip-firing move that places a nonzero weight on $u_i$ may be followed by a move that moves this weight to $v_i$. Thus we can treat the cycle $G_n$ in isolation.

We will show that $nv_1\sim nv_n$ from which the result immediately follows. To see this, we perform the following sequence of chip firing moves. Fire $v_1$ $2n$-times, fire $v_2$ $(2n-1)$-times, and so on. In other words, we fire the vertex $v_\ell$ $(2n-\ell+1)$-times. It is elementary to check that the resulting divisor has nonzero chips only on $v_n$. Since chip firing preserves degree, we see that this sequence of moves yields $nv_n$, and the result follows.
\end{proof}

We will also need the following two facts.

\begin{lemma}\label{lem4}
Let $G$ be defined by attaching three copies of $A_m$ or $B_m$ to the degree 2 vertices of $K_{2,3}$. For any $n\in \mathbb{Z}$ and choices $v$ and $w$ of degree 2 vertices in $K_{2,3}$, the divisors $D=(2n)v$ and $D'=(2n)w$ are equivalent.
\end{lemma}

\begin{proof}
It suffices to show the $n = 1$ case, so take $D = 2(v_1)$. Denote by $w_1,w_2,w_3$ the degree three vertices of $K_{2,3}$. Fire $v_1$ to obtain an equivalence
\[
D\sim (w_1)+(w_2)-(v_1)+(u_1),
\]
where $v$ is the vertex of $A_m$ or $B_m$ adjacent to $v_1$. As in Lemma~\ref{lem1}, since $u_1$ and $v_1$ are attached by a bridge edge, we see that $v_1\sim v$ and thus, $D\sim w_1+w_2$. Now fire $v_2$ and use this again to conclude $D\sim 2(v_2)$.
\end{proof}

The proof of the following lemma follows similarly to the one above.

\begin{lemma}\label{lem5}
Let $G$ be a graph defined by joining together two triangles via the identification of an edge $(w_1,w_2)$. Say $v_1$ and $v_2$ are the other vertices of $G$. Then, if $D=(2n)v_1$ and $D' = (2n)v_2$, $D\sim D'$.
\end{lemma}

We are now ready to prove Theorem~\ref{thm: simple-graphs}. In~\cite{OV10}, van Opstall and Veliche prove sharp numerical bounds for the size of the automorphism group of a genus $g$ trivalent graph. The authors show that these bounds are sharp by explicitly constructing a maximally symmetric simple graph $C_g''$. These graphs are constructed in similar fashion to the cases of multigraphs. We now show that this graph $C_g''$ is special for sufficiently large genus. For an explicit construction of the graphs $C_g''$ we refer the reader to the aforementioned paper. \\

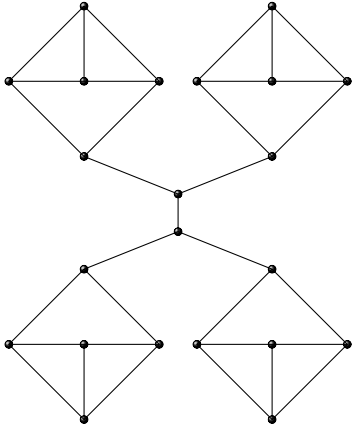
\begin{figure}[h!]
\begin{tikzpicture}
\coordinate (A) at (0,1);
\coordinate (B) at (-1,0);
\coordinate (C) at (0,-1);
\coordinate (D) at (1,0);
\coordinate (E) at (0,0);

\draw[ball color=black] (A) circle (0.5mm);
\draw[ball color=black] (B) circle (0.5mm);
\draw[ball color=black] (C) circle (0.5mm);
\draw[ball color=black] (D) circle (0.5mm);
\draw[ball color=black] (E) circle (0.5mm);
\draw (A)--(B)--(C)--(D)--(A); \draw (E)--(C); \draw (E)--(B); \draw (E)--(D);

\coordinate (A') at (2.5,1);
\coordinate (B') at (1.5,0);
\coordinate (C') at (2.5,-1);
\coordinate (D') at (3.5,0);
\coordinate (E') at (2.5,0);

\draw[ball color=black] (A') circle (0.5mm);
\draw[ball color=black] (B') circle (0.5mm);
\draw[ball color=black] (C') circle (0.5mm);
\draw[ball color=black] (D') circle (0.5mm);
\draw[ball color=black] (E') circle (0.5mm);
\draw (A')--(B')--(C')--(D')--(A'); \draw (E')--(C'); \draw (E')--(B'); \draw (E')--(D');

\coordinate (A'') at (2.5,4.5);
\coordinate (B'') at (1.5,3.5);
\coordinate (C'') at (2.5,2.5);
\coordinate (D'') at (3.5,3.5);
\coordinate (E'') at (2.5,3.5);

\draw[ball color=black] (A'') circle (0.5mm);
\draw[ball color=black] (B'') circle (0.5mm);
\draw[ball color=black] (C'') circle (0.5mm);
\draw[ball color=black] (D'') circle (0.5mm);
\draw[ball color=black] (E'') circle (0.5mm);
\draw (A'')--(B'')--(C'')--(D'')--(A''); \draw (E'')--(A''); \draw (E'')--(B''); \draw (E'')--(D'');

\coordinate (A''') at (0,4.5);
\coordinate (B''') at (-1,3.5);
\coordinate (C''') at (0,2.5);
\coordinate (D''') at (1,3.5);
\coordinate (E''') at (0,3.5);

\draw[ball color=black] (A''') circle (0.5mm);
\draw[ball color=black] (B''') circle (0.5mm);
\draw[ball color=black] (C''') circle (0.5mm);
\draw[ball color=black] (D''') circle (0.5mm);
\draw[ball color=black] (E''') circle (0.5mm);
\draw (A''')--(B''')--(C''')--(D''')--(A'''); \draw (E''')--(A'''); \draw (E''')--(B'''); \draw (E''')--(D''');

\coordinate (X) at (1.25,2);
\coordinate (Y) at (1.25,1.5);

\draw[ball color=black] (X) circle (0.5mm);
\draw[ball color=black] (Y) circle (0.5mm);

\draw (A)--(Y)--(A');\draw (Y)--(X);\draw (C'')--(X)--(C''');

\end{tikzpicture}
\caption{The graph $C_{12}''$.}
\label{fig: pinched}

\end{figure}

\noindent\textit{Proof of Theorem~\ref{thm: simple-graphs}.} In each case, components $A_m$ or $B_m$ occur in $C_g''$; thus, the tree $T$ appears as well.

\textit{Case 1.} Suppose that $C''_g$ consists of some collection of $A_m$'s or $B_m$'s joined at a common root, a common edge or path, by a $K_{2,3}$, by a square, or by two triangles that share a common edge. Define $D=4v$ for any vertex $v\in T$. $D$ has rank at least one. This is shown by Lemmas~\ref{lem1} and~\ref{lem2}, together with one of Lemmas~\ref{lem3},~\ref{lem4} or~\ref{lem5} depending on whether a square, $K_{2,3}$ or two triangles appear in the graph. Furthermore, $\rho_4^1(g)<0$ when $g>6$ as is the case here, whence the result follows.

\textit{Case 2.} Suppose that $C''_g$ consists of some collection of $A_m$'s or $B_m$'s joined to the vertices of a triangle or to a path in which a triangle appears at one end. Let $v$ and $w$ be two distinct vertices of the triangle. Define $D=3v+2w$. Lemmas~\ref{lem1} to~\ref{lem3} guarantee that $D$ has rank at least 1. Furthermore, $\rho_5^1(g)<0$ when $g>8$ as is the case here, whence the result follows.

\textit{Case 3.} Suppose that $C''_g$ consists of some collection of $A_m$'s or $B_m$'s joined to the vertices of a pentagon. Define $D=5v$ for some vertex $v\in T$. Then, Lemmas~\ref{lem1} to~\ref{lem3} guarantee that $D$ has rank at least 1. Furthermore, $\rho_5^1(g)<0$ when $g>8$ as is the case here, whence the result follows.

Together these comprise all possible configurations for $C_g''$, and the result follows. \qed

\bibliographystyle{plain}
\bibliography{MaxSymmetricGraphs}
\nocite{*}
\end{document}